\documentclass[10pt]{amsart}
\usepackage{ amssymb }
\usepackage{mathrsfs}
\usepackage[colorlinks,linkcolor=blue,citecolor=blue,urlcolor=blue,bookmarks=false,hypertexnames=true]{hyperref}
\usepackage[margin=0.9in]{geometry}
\usepackage{ amsmath}
\usepackage{graphicx} 
\DeclareSymbolFont{largesymbolsA}{U}{txexa}{m}{n}
\DeclareMathSymbol{\varprod}{\mathop}{largesymbolsA}{16}
\newtheorem{theorem}{Theorem}[section]

\newtheorem{lem}[theorem]{Lemma}

\newtheorem{cor}{Corollary}[theorem]
\theoremstyle{definition}

\theoremstyle{remark}
\newtheorem{remark}[theorem]{Remark}

\numberwithin{equation}{section}

\begin{document}
\author{Ankush Goswami}
\address{Research Institute for Symbolic Computation (RISC), JKU, Linz.}
\email{ankushgoswami3@gmail.com, ankush.goswami@risc.jku.at}
\title[Congruences for generalized Fishburn numbers at roots of unity]{Congruences for generalized Fishburn numbers at roots of unity}
\thanks{}


\date{}

\begin{abstract}
There has been significant recent interest in the arithmetic
properties of the coefficients of $F(1-q)$ and $\mathscr{F}_t(1-q)$
where $F(q)$ is the Kontsevich-Zagier strange series and 
 $\mathscr{F}_t(q)$ is the strange series associated to a family 
of torus knots as studied by Bijaoui, Boden, Myers, Osburn, Rushworth, Tronsgard
and Zhou. In this paper, we prove prime power congruences for two families of generalized Fishburn numbers, namely, for the coefficients of 
$(\zeta_N - q)^s F((\zeta_N - q)^r)$ and $(\zeta_N - q)^s \mathscr{F}_t((\zeta_N -
q)^r)$, where $\zeta_N$ is an $N$th root of unity and $r$, $s$ are certain integers.
\end{abstract}
\maketitle
\section{Introduction}
In \cite{AS}, Andrews and Sellers studied congruence properties for $\xi(n)$ which are the coefficients in the $q$-series expansion
\begin{eqnarray}\label{Fish}
F(1-q)=:\sum_{n\geq 0}\xi(n)q^n=1+q+2q^2+5q^3+15q^4+53q^5+\cdots
\end{eqnarray}
where $F(q):=\sum_{n\geq 0}(q)_n$ is the Kontsevich-Zagier ``strange" series \cite{Z} and 
\begin{eqnarray}
(q)_n=(q;q)_n:=\prod_{k=1}^n(1-q^k),
\end{eqnarray}
is the standard $q$-hypergeometric notation, valid for $n\in\mathbb{N}\cup\{\infty\}$. Here the term ``strange" is used to indicate the fact that $F(q)$ does not converge anywhere inside or outside the unit disk, but is well-defined when $q$ is a root of unity (where it is finite). If $\zeta_N$ is an $N$th root of unity, then 
\begin{eqnarray}\label{con1}
F(\zeta_N-q)=\sum_{i=0}^{N-1}(\zeta_N-q;\zeta_N-q)_i+\sum_{i=0}^{N-1}\sum_{n=1}^\infty (\zeta_N-q;\zeta_N-q)_{nN+i}.
\end{eqnarray}
Noting that $(\zeta_N-q;\zeta_N-q)_{nN+i}=O(q^n)$, we immediately see that the sum in the right-hand side of (\ref{con1}) is absolutely convergent for $|q|<1$. Thus $F(\zeta_N-q)$ and in particular, $F(1-q)$ is well-defined.

Recently, Garvan \cite{G} studied congruence properties for what he called $r$-Fishburn numbers which are coefficients in the $q$-series expansion of $F((1-q)^r)=:\sum_{n\geq 0}\xi_r(n)q^n$ (for $r\in\mathbb{Z}\setminus\{0\}$) and Straub \cite{Straub} extended these congruences to prime power modulus which, in the case $r=1$, were conjectured by Andrews and Sellers \cite{AS}. For example, we have
\begin{eqnarray}
&\xi(5^\ell n-1)\equiv \xi(5^\ell n-2)\equiv 0\;(\mbox{mod}\;5^\ell),&\\
&\xi_2(7^\ell n-1)\equiv \xi_2(7^\ell n-2)\equiv \xi_2(7^\ell n-3)\equiv 0\;(\mbox{mod}\;7^\ell),&\\
&\xi_3(19^\ell n-1)\equiv 0\;(\mbox{mod}\;19^\ell)&
\end{eqnarray}
for all natural numbers $\ell$ and $n$.

The proofs of the aforementioned results rely on a crucial divisibility property of the dissections of the partial sums of the corresponding $q$-series by a power of $(1-q)$. In the case for $F(q)$, this property was first proved by Andrews and Sellers \cite{AS} and strengthened to a ``strong divisibility property'' by Ahlgren and Kim \cite[Theorem 1, pp. 2]{AK}, conjectured earlier by Andrews and Sellers. Later, Ahlgren, Kim and Lovejoy \cite{AKL} obtained a strong divisibility property for two generic families of $q$-hypergeometric series, which, like the Kontesevich-Zagier series are also strange.

Consider the \textit{Kontsevich-Zagier} series associated to the family of torus knots $T(3,2^t)$, $t\geq 2$, as follows:
\begin{eqnarray}\label{KZ}
\mathscr{F}_t(q)&:=&(-1)^{h''(t)}q^{-h'(t)}\sum_{n\geq 0}(q)_n\sum_{3\sum_{\ell=1}^{m(t)-1}j_\ell\ell\equiv 1\;(\text{mod}\;m(t))}q^{\frac{-a(t)+\sum_{\ell=1}^{m(t)-1}j_\ell\ell}{m(t)}+\sum_{\ell=1}^{m(t)-1}\binom{j_\ell}{2}}\notag\\&&\times \sum_{k=0}^{m(t)-1}\prod_{\ell=1}^{m(t)-1}\binom{n+I(\ell\leq k)}{j_\ell}_q
\end{eqnarray}
where 
\begin{eqnarray*}
h''(t)=\left\{\begin{array}{cc}
\frac{2^t-1}{3},&\mbox{if}\;t\;\mbox{is even},\\
\frac{2^t-2}{3},&\mbox{if}\;t\;\mbox{is odd},
\end{array}\right.\hspace{0.5cm}
h'(t)=\left\{\begin{array}{cc}
\frac{2^t-4}{3},&\mbox{if}\;t\;\mbox{is even},\\
\frac{2^t-5}{3},&\mbox{if}\;t\;\mbox{is odd},
\end{array}\right.\hspace{0.5cm}
a(t)=\left\{\begin{array}{cc}
\frac{2^{t-1}+1}{3},&\mbox{if}\;t\;\mbox{is even},\\
\frac{2^{t}+1}{3},&\mbox{if}\;t\;\mbox{is odd},
\end{array}\right.
\end{eqnarray*}
$m(t)=2^{t-1}$, $I(*)$ is the characteristic function and 
\begin{eqnarray*}
\binom{n}{k}_q=\dfrac{(q)_n}{(q)_{n-k}(q)_k}
\end{eqnarray*}
is the $q$-binomial coefficient. We note here that the $t=1$ case yields
\begin{eqnarray}\label{sKV}
\mathscr{F}_1(q)=F(q),
\end{eqnarray}
since the two inner sums in the right-hand side of (\ref{KZ}) are empty which by convention can be assumed to be $1$. Thus $\mathscr{F}_t(q)$ is a generalized Kontsevich-Zagier series and as with the original Kontsevich-Zagier series, $\mathscr{F}_t(q)$ converges in a similar manner. 

Also we have (see \cite[Eq. 1.5, pp. 3]{Be})
\begin{eqnarray}\label{Jones}
\zeta_N^{2^t-1}\mathscr{F}_t(\zeta_N)=J_{N}(T(3,2^t);\zeta_N)
\end{eqnarray}
where for a knot $K$, $J_N(K;q)$ is the usual colored Jones polynomial, normalized to be $1$ for the unknot. It is a knot invariant and thus plays a leading role in many open problems in quantum topology (for more details, see \cite{Be}). Thus, in view of (\ref{sKV}) and (\ref{Jones}) we have
\begin{eqnarray}
\zeta_NF(\zeta_N)=J_N(T(3,2);\zeta_N).
\end{eqnarray}
Thus it is of much interest to study congruences of coefficients for the more general $q$-series $\mathscr{F}_t(q)$. We define
\begin{eqnarray}
\mathscr{F}_t(1-q)=:\sum_{n\geq 0}\xi_t(n)q^n.
\end{eqnarray}
For natural numbers $m$ and $t\geq 2$ and the periodic function
\begin{eqnarray}
\chi_t(n)=\chi_{3\cdot 2^{t+1}}(n):=\left\{
\begin{array}{ccc}
    1 & \mbox{if}\;n\equiv 2^{t+1}-3, 3+2^{t+2}\;(\mbox{mod}\;3\cdot 2^{t+1}), \\
    -1 & \mbox{if}\;n\equiv 2^{t+1}+3, 2^{t+2}-3\;(\mbox{mod}\;3\cdot 2^{t+1}),\\
    0&\mbox{otherwise},
\end{array}\right.
\end{eqnarray}
we define the set 
\begin{eqnarray}
S_{t,\chi_t}(m):=\left\{0\leq j\leq m-1: j\equiv \dfrac{n^2-(2^{t+1}-3)^2}{3\cdot 2^{t+2}}\;(\mbox{mod}\;m)\;\mbox{where}\;\chi_t(n)\neq 0\right\}.
\end{eqnarray}
In \cite{Be}, the following result was proved for the \textit{generalized Fishburn numbers} $\xi_t(n)$:
\begin{theorem}\label{BeO}
For a prime $p\geq 5$, let $j\in\{1,2,3,\cdots p-1-S_{t,\chi_t}(p)\}$. Then 
\begin{eqnarray*}
\xi_t(p^\ell n-j)\equiv 0\;(\emph{mod}\;p^\ell)
\end{eqnarray*}
for all natural numbers $\ell, n$ and $t\geq 2$. 
\end{theorem}
The aim of this article is to obtain prime power congruences for coefficients in the $(\zeta_N-q)$ expansions of $F(q)$ and $\mathscr{F}_t(q)$, where for $N\in\mathbb{N}$, $\zeta_N$ is an $N$th root of unity. More generally, we obtain prime power congruences for the coefficients in the $q$-expansions of $(\zeta_N-q)^sF((\zeta_N-q)^r)$ and $(\zeta_N-q)^s\mathscr{F}_t((\zeta_N-q)^r)$ for certain choices of integers $r$ and $s$. 

The rest of the paper is organized as follows. In Section \ref{NC}, we introduce some notations and conventions. In Section \ref{MR}, we state our main results and also discuss a few important remarks. In Section \ref{PR}, we prove some basic results and also record some preliminaries from \cite{AKL, Be, G, Straub}. Finally, in Section \ref{PM} we give the proofs of the main results.
\section{Notations and Conventions}\label{NC}
Throughout $p$ will denote a prime $\geq 5$, $t$ a natural number $\geq 2$. Let $r, s$ be any integers. For a natural number $N$, let $\zeta_N$ be an $N$th root of unity. Define the coefficients $\xi_{r,N}(n)$ and $\xi_{r,N,t}(n)$ by
\begin{eqnarray}\label{gFT1}
F((\zeta_N-q)^r)=:\sum_{n\geq 0}\xi_{r,N}(n)q^n,\hspace{0.5cm}\mathscr{F}_t((\zeta_N-q)^r)=:\sum_{n\geq 0}\xi_{r,N,t}(n)q^n
\end{eqnarray}
and the coefficients $\xi_{r,s,N}(n)$ and $\xi_{r,s,N,t}(n)$ by
\begin{eqnarray}\label{gFT2}
(\zeta_N-q)^sF((\zeta_N-q)^r)=:\sum_{n\geq 0}\xi_{r,s,N}(n)q^n,\hspace{0.5cm}
(\zeta_N-q)^s\mathscr{F}_t((\zeta_N-q)^r)=:\sum_{n\geq 0}\xi_{r,s,N,t}(n)q^n.
\end{eqnarray}
Then (\ref{gFT1}) and (\ref{gFT2}) imply that
\begin{eqnarray}
\xi_{r,s,N}(n)=\sum_{j=0}^s(-1)^j\zeta_N^{s-j}\binom{s}{j}\xi_{r,N}(n-j),\hspace{0.3cm}\xi_{r,s,N,t}(n)=\sum_{j=0}^s(-1)^j\zeta_N^{s-j}\binom{s}{j}\xi_{r,N,t}(n-j).
\end{eqnarray}
It is clear that $\xi_{1,0,1,t}(n)=\xi_{t}(n)$. Define the following sets:
\begin{eqnarray}
S(p,r,s)&:=&\left\{0\leq j\leq p-1: j-s\equiv r\cdot\dfrac{n(3n-1)}{2}\;(\mbox{mod}\;p)\right\},\notag\\
S^*(p,r,s)&:=&\left\{0\leq j\leq p-1: j-s\equiv r\cdot\dfrac{n(3n-1)}{2}\;(\mbox{mod}\;p),\;\; 24(j-s)\not\equiv -r\;(\mbox{mod}\;p)\right\},\notag\\
S_{t,\chi_t}(p,r,s)&:=&\left\{0\leq j\leq p-1: j-s\equiv r\cdot\dfrac{n^2-(2^{t+1}-3)^2}{3\cdot 2^{t+2}}\;(\mbox{mod}\;p)\;\mbox{where}\;\chi_t(n)\neq 0\right\},\notag\\
S_{t,\chi_t}^*(p,r,s)&:=&\Bigg\{j\in S_{t,\chi_t}(p,r,s): 3\cdot 2^{t+2}(j-s)\not\equiv -r\cdot (2^{t+1}-3)^{2}\;(\mbox{mod}\;p)\Bigg\}.
\end{eqnarray}
Next, let $b_{n,t}(\zeta_p)$ and $c_{n,t}(\zeta_p)$ be the coefficients in the power series expansions of the following: 
\begin{eqnarray}\label{coeff1}
e^{\frac{-u(2^{t+1}-3)^2}{3\cdot 2^{t+2}}}\mathscr{F}_t(\zeta_p e^{-u})=\sum_{n=0}^\infty \dfrac{c_{n,t}(\zeta_p)}{n!}\left(\dfrac{u}{3\cdot 2^{t+2}}\right)^n,\hspace{0.5cm}\mathscr{F}_t(\zeta_p e^{-u})=\sum_{n=0}^\infty \dfrac{b_{n,t}(\zeta_p)}{n!}u^n.
\end{eqnarray}
Let $F(q;N)$ and $\mathscr{F}_t(q;N)$ denote the truncation of $F(q)$ and $\mathscr{F}_t(q)$, respectively at a height $N$ defined by 
\begin{eqnarray}\label{trunc}
\mathscr{F}_t(q;N)&=&(-1)^{h''(t)}q^{-h'(t)}\sum_{n= 0}^N(q)_n\sum_{3\sum_{\ell=1}^{m(t)-1}j_\ell\ell\equiv 1\;(\text{mod}\;m(t))}q^{\frac{-a(t)+\sum_{\ell=1}^{m(t)-1}j_\ell\ell}{m(t)}+\sum_{\ell=1}^{m(t)-1}\binom{j_\ell}{2}}\notag\\&&\times \sum_{k=0}^{m(t)-1}\prod_{\ell=1}^{m(t)-1}\binom{n+I(\ell\leq k)}{j_\ell}_q,\notag\\
F(q;N)&=&\sum_{n=0}^N(q)_n.
\end{eqnarray}
Consider the $p$-dissection of $F(q;N)$ and $\mathscr{F}_t(q;N)$ as follows:
\begin{eqnarray}\label{p-dis}
F(q;N)=\sum_{i=0}^{p-1}q^iA_{p}(N,i,q^p),\hspace{0.5cm}\mathscr{F}_t(q;N)=\sum_{i=0}^{p-1}q^i\mathcal{A}_{p,t}(N,i,q^p).
\end{eqnarray}
Also, consider the following:
\begin{eqnarray}\label{sum1}
A_p(pn-1,i,1-q)=\sum_{k\geq 0}\alpha(p,n,i,k)q^k,\hspace{0.5cm}\mathcal{A}_{p,t}(pn-1,i,1-q)=\sum_{k\geq 0}\alpha_t(p,n,i,k)q^k.
\end{eqnarray}
Define the following set of primes:\\ \\
\textbf{(I)} \hspace{0.1cm}$\mathcal{P}_1$ is the set of primes satisfying
\begin{eqnarray*}\label{P1}
p\equiv 1\;(\mbox{mod}\;3\cdot 2^{t+1}).
\end{eqnarray*}
\textbf{(II)} \hspace{0.1cm}$\mathcal{P}_2$ is the set of primes satisfying
\begin{eqnarray*}\label{P2}
p\equiv 3\cdot 2^{t+1}-1\;(\mbox{mod}\;3\cdot 2^{t+1}).
\end{eqnarray*}
\textbf{(III)} \hspace{0.1cm}$\mathcal{P}_3$ is the set of primes satisfying
\begin{eqnarray*}\label{P3}
p\equiv r_1(t)\;\;(\mbox{mod}\;3\cdot 2^{t+1}),\;r_{1}(t)=\left\{\begin{array}{cc}2^{t+1}-1,& \mbox{if}\;t\;\mbox{is even},\\
2^{t+1}+1,& \mbox{if}\;t\;\mbox{is odd}.
\end{array}\right.
\end{eqnarray*}
\textbf{(IV)} \hspace{0.1cm}And $\mathcal{P}_4$ is the set of primes satisfying
\begin{eqnarray*}\label{P4}
p\equiv r_2(t)\;\;(\mbox{mod}\;3\cdot 2^{t+1}),\;r_{2}(t)=\left\{\begin{array}{cc}2^{t+2}+1,& \mbox{if}\;t\;\mbox{is even},\\
2^{t+2}-1,& \mbox{if}\;t\;\mbox{is odd}.
\end{array}\right.
\end{eqnarray*}
Recall that every $n\in\mathbb{Q}$ has a unique $p$-adic expansion 
\begin{eqnarray}
n=\sum_{k=\nu_p(n)}^\infty n_kp^{k},
\end{eqnarray}
where $n_k\in\{0,1,\cdots, p-1\}$ and $\nu_p(n)$ is the $p$-adic valuation of $n$. We denote by $\mbox{digit}_k(n;p)=n_k.$ Finally, if $r_{\zeta_N}\in\mathbb{Z}[\zeta_N]$ and $\lambda\in\mathbb{N}$, then by $r_{\zeta_N}\equiv 0\;(\mbox{mod}\;p^\lambda)$ we mean that there exists an $r'_{\zeta_N}\in\mathbb{Z}[\zeta_N]$ such that $r_{\zeta_N}=p^\lambda\cdot r'_{\zeta_N}$. 
\section{Main results}\label{MR}
\begin{theorem}\label{main1}
Let $p\geq 5$ be a prime, $N\in\mathbb{N}$ and $r, s\in\mathbb{Z}$ be such that $p\nmid r$. Then, for all $\lambda, m\in\mathbb{N}$ and $j\in\left\{1,2,\cdots, p-1-\max{S(p,r,s)}\right\}$ we have
\begin{eqnarray*}
\xi_{r,s,N}(p^\lambda m-j)\equiv 0\;(\emph{mod}\;p^\lambda).
\end{eqnarray*}
Further, if $N\mid rp$ and the triple $(p,r,s)$ satisfies 
\begin{eqnarray}\label{c1}
\emph{digit}_1(s-r/24;p)< p-1,
\end{eqnarray}
then the set $S(p,r,s)$ can be replaced by $S^*(p,r,s)$. 
\end{theorem}
\begin{theorem}\label{main2}
Let $t\geq 2$ and $p\geq 5$ be a prime. Let $N\in\mathbb{N}$ and $r, s\in\mathbb{Z}$ be such that $p\nmid r$. Set $a=(2^{t+1}-3)^2$ and $b=3\cdot 2^{t+2}$. Then, for all $\lambda, m\in\mathbb{N}$ and $j\in\left\{1,2,\cdots, p-1-\max{S_{t,\chi_t}(p,r,s)}\right\}$ we have
\begin{eqnarray*}
\xi_{r,s,N,t}(p^\lambda m-j)\equiv 0\;(\emph{mod}\;p^\lambda).
\end{eqnarray*}
Moreover, let $p\in\mathcal{P}_1\cup\mathcal{P}_2\cup\mathcal{P}_3\cup\mathcal{P}_4$ where $\mathcal{P}_1, \mathcal{P}_2, \mathcal{P}_3$ and $\mathcal{P}_4$ are as in \emph{\textbf{(I)} - \textbf{(IV)}}. If $N\mid rp$ and the triple $(p,r,s)$ satisfies
\begin{eqnarray}\label{c4}
\emph{digit}_1(s-ar/b;p)<p-1,
\end{eqnarray}
then the set $S_{t,\chi_t}(p,r,s)$ can be replaced by $S^{*}_{t,\chi_t}(p,r,s)$.  
\end{theorem}
\begin{cor}\label{main3}
Let $p\geq 5$ be a prime. Let $N\in\mathbb{N}$ and $r, s\in\mathbb{Z}$ be such that $p\nmid r$. Then, for all $\lambda, m\in\mathbb{N}$ and $j\in\left\{1,2,\cdots, p-1-\max{S_{2,\chi_2}(p,r,s)}\right\}$ we have
\begin{eqnarray*}
\xi_{r,s,N,2}(p^\lambda m-j)\equiv 0\;(\emph{mod}\;p^\lambda).
\end{eqnarray*}
Moreover if $p\equiv \pm 1, \pm 7\;(\emph{mod}\;24)$, $N\mid rp$ and the triple $(p,r,s)$ satisfies
\begin{eqnarray}\label{c4}
\emph{digit}_1(s-25r/48;p)<p-1,
\end{eqnarray}
then the set $S_{2,\chi_2}(p,r,s)$ can be replaced by $S^{*}_{2,\chi_2}(p,r,s)$.
\end{cor}
\begin{remark}
We note here that Theorem 1.5 in \cite{Straub} is a special case ($N=1$) of Theorem \ref{main1}. Theorem 1.1 in \cite{Be} is a special case of Theorem \ref{main2} ($N=1, r=1, s=0$). Besides for primes $p\in\mathcal{P}_1\cup\mathcal{P}_2\cup\mathcal{P}_3\cup\mathcal{P}_4$, Theorem \ref{main2} improves Theorem 1.1 in \cite{Be} in the case $N=1, r=1, s=0$ for all $t\geq 2$, by replacing the set $S_{t,\chi_t}(p,r,s)$ by the smaller set $S^*_{t,\chi_t}(p,r,s)$.
\end{remark}
\begin{remark}\label{rem2}
As an example, take $(r,s,N)=(1,0,2)$. This triple corresponds to the coefficients in the $q$-series expansion of $F(-1-q)$. In view of Theorem \ref{main1}, we find
\begin{eqnarray}
\xi_{1,0,2}(5^\lambda m-1)\equiv \xi_{1,0,2}(5^\lambda m-2)&\equiv &0\;(\mbox{mod}\;5^\lambda),\notag\\
\xi_{1,0,2}(7^\lambda m-1)&\equiv & 0\;(\mbox{mod}\;7^\lambda),\notag\\
\xi_{1,0,2}(11^\lambda m-1)\equiv \xi_{1,0,2}(11^\lambda m-2)\equiv \xi_{1,0,2}(11^\lambda m-3)&\equiv & 0\;(\mbox{mod}\;11^\lambda),\notag\\
\xi_{1,0,2}(17^\lambda m-1)&\equiv & 0\;(\mbox{mod}\;17^\lambda),\notag\\
\xi_{1,0,2}(19^\lambda m-1)\equiv \xi_{1,0,2}(19^\lambda m-2)\equiv \xi_{1,0,2}(19^\lambda m-3)&\equiv & 0\;(\mbox{mod}\;19^\lambda),
\end{eqnarray}
which are exactly the same prime power congruences in \cite{Straub} for the coefficients in the $q$-series expansion of $F(1-q)$. In fact, as is easily seen, the first part of Theorem \ref{main1} implies that, for fixed choices of integers $r$ and $s$, the congruences for the coefficients $\xi_{r,s,N}(n)$ are independent of the choices of $N$ (or roots of unity, $\zeta_N$), since the set $S(p,r,s)$ is independent of $N$. However, the second part of Theorem \ref{main1} depends on the choice of $\zeta_N$. Thus, for the triple $(r,s,N)=(1,0,2)$, we do not get any additional congruences (contrary to the case $N=1$) than the ones obtained from the first part of the theorem. In fact, we find that
\begin{eqnarray*}
\xi_{1,0,2}(23\cdot 1-1)&=&-3374324885490973100341136883972043\not\equiv 0\;(\mbox{mod}\;23),\notag\\\xi_{1,0,2}(23\cdot 1-2)&=&-47914053901185858013549651979546\not\equiv 0\;(\mbox{mod}\;23)\notag\\
\xi_{1,0,2}(23\cdot 1-3)&=&184970580844281275291492442891\not\equiv 0\;(\mbox{mod}\;23),\notag\\
\xi_{1,0,2}(23\cdot 1-4)&=&37985301942246535793275853285\not\equiv 0\;(\mbox{mod}\;23),\notag\\
\xi_{1,0,2}(23\cdot 1-5)&=&1362966752518988604618378515\not\equiv 0\;(\mbox{mod}\;23).\notag
\end{eqnarray*}
We note here that these numbers are already very large. We also note that for the triple $(r,s,N)=(2,0,2)$, all the conditions in the second part of Theorem \ref{main2} are satisfied and we get the following additional congruences for coefficients of $F((1+q)^2)$, not obtained from the first part of the theorem. 
\begin{eqnarray*}
&\xi_{2,0,2}(7^\lambda m-3)\equiv 0\;(\mbox{mod}\;7^\lambda),&\notag\\
&\xi_{2,0,2}(11^\lambda m-1)\equiv \xi_{2,0,2}(11^\lambda m-2)\equiv 0\;(\mbox{mod}\;11^\lambda).&
\end{eqnarray*}
\end{remark}
\begin{remark}
For $t\geq 2$, as in Remark \ref{rem2}, the first part of Theorem \ref{main2} yields the same congruences irrespective of the choices of roots of unity. However, the second part of the theorem depends on the choices of roots of unity. As an example, let us look at Corollary \ref{main3}. This corresponds to the triple $(r,s,N,t)=(1,0,1,2)$ in Theorem \ref{main2}. We see from the second part of Corollary \ref{main3} that $p\in\mathcal{P}_1\cup\mathcal{P}_2\cup\mathcal{P}_3\cup\mathcal{P}_4$ if and only if $p\equiv\pm 1,\;\pm 7\;(\mbox{mod}\;24)$. Also, the conditions $N\mid rp$ and $\mbox{digit}_1(-25/48;p)=\lfloor25p/48\rfloor<p-1$ are satisfied, and we find the following additional congruences modulo $23$, not obtained from Theorem \ref{BeO} in the case $t=2$. 
\begin{eqnarray*}
&\scalebox{0.95}{$\xi_{1,0,1,2}(23^\lambda m-5)\equiv \xi_{1,0,1,2}(23^\lambda m-4) \equiv \xi_{1,0,1,2}(23^\lambda m-3)\equiv\xi_{1,0,1,2}(23^\lambda m-2)\equiv \xi_{1,0,1,2}(23^\lambda m-1)\equiv 0\;(\mbox{mod}\;23^\lambda)$}.&
\end{eqnarray*}
\end{remark}
\section{Preliminary results}\label{PR}
\begin{lem}\label{Coeffs}
Let $N_t=3\cdot 2^{t+1}\cdot p$. Then we have
\begin{eqnarray*}
c_{n,t}(\zeta_p)&=&\dfrac{(-1)^n N_t^{2n+1}}{2n+2}\sum_{m=1}^{N_t/2}\zeta_p^{\frac{m^2-(2^{t+1}-3)^2}{3\cdot 2^{t+2}}}\chi_t(m)\;B_{2n+2}(m/N_t),\\
b_{n,t}(\zeta_p)&=&\dfrac{(2^{t+1}-3)^{2n}}{(3\cdot 2^{t+2})^n}\sum_{j=0}^n\binom{n}{j}\sum_{i=0}^{p-1}\gamma_t(j,i)\zeta_p^i,
\end{eqnarray*}
where $B_n(x)$ is the $n$th Bernoulli polynomial and 
\begin{eqnarray*}
\gamma_t(j,i)=\dfrac{(-1)^jN_t^{2j+1}}{(2j+2)(2^{t+1}-3)^{2j}}\sum_{\substack{1\leq m\leq N_t/2\\(m^2-(2^{t+1}-3)^2)/3\cdot 2^{t+2}\equiv i\;(\emph{mod}\;p)}}\chi_t(m)B_{2j+2}(m/N_t).
\end{eqnarray*}
\end{lem}
\begin{proof}
Let us define the following:
\begin{eqnarray}
f(u)=e^{\frac{-u(2^{t+1}-3)^2}{3\cdot 2^{t+2}}}\mathcal{P}_t(\zeta_p e^{-u}):=-\dfrac{1}{2}\sum_{n\geq 0}n\chi_t(n)\zeta_p^{\frac{n^2-(2^{t+1}-3)^2}{3\cdot 2^{t+2}}}e^{-\frac{n^2u}{3\cdot 2^{t+2}}}.
\end{eqnarray}
Set $C_t(n)=\zeta_p^{\frac{n^2-(2^{t+1}-3)^2}{3\cdot 2^{t+2}}}\chi_t(n)$, and define the $L$-function
\begin{eqnarray}
L(2s-1,C_t)=\sum_{n=1}^\infty \dfrac{C_t(n)}{n^{2s-1}},\hspace{1cm}\mbox{Re}(s)>\dfrac{1}{2}.
\end{eqnarray}
Then since $C_t: \mathbb{Z}\rightarrow \mathbb{C}$ is a function with period $N_t$ and mean value $0$ (see \cite[prop. 2.2]{Be}), the result in \cite[Lemma 3.2, pp. 7]{AKL} implies that 
\begin{eqnarray}\label{LBer}
L(-2n-1,C_t)=\dfrac{-N_t^{2n+1}}{2n+2}\sum_{m=1}^{N_t}C_t(m)B_{2n+2}(m/N_t).
\end{eqnarray}
Now by Mellin transform (see \cite{MV} for more details), we have
\begin{eqnarray}\label{MT}
\int_0^\infty u^{s-1}f(u)du&=&-\dfrac{1}{2}\sum_{n\geq 1}n\chi_t(n)\zeta_p^{\frac{n^2-(2^{t+1}-3)^2}{3\cdot 2^{t+2}}}\int_0^\infty u^{s-1}e^{-\frac{n^2u}{3\cdot 2^{t+2}}}du\notag\\
&=&-\dfrac{1}{2}\sum_{n\geq 1}n\chi_t(n)\zeta_p^{\frac{n^2-(2^{t+1}-3)^2}{3\cdot 2^{t+2}}}\left(\frac{n^2}{3\cdot 2^{t+2}}\right)^{-s}\Gamma(s)\notag\\
&=&-\dfrac{(3\cdot 2^{t+2})^{s}}{2}\Gamma(s)\sum_{n\geq 1}\dfrac{C_t(n)}{n^{2s-1}}=-\dfrac{(3\cdot 2^{t+2})^{s}}{2}\Gamma(s)L(2s-1,C_t).
\end{eqnarray}
Applying Mellin inversion in (\ref{MT}) we get
\begin{eqnarray}\label{fu}
f(u)=-\dfrac{1}{2}\left(\dfrac{1}{2i\pi}\int_{\text{Re}(s)=c}(3\cdot 2^{t+2})^{s}\Gamma(s)L(2s-1,C_t)u^{-s}ds\right),
\end{eqnarray}
where $c>1/2$. Next, we note that 
\begin{eqnarray}\label{Hu}
L(2s-1,C_t)=\dfrac{1}{N_t^{2s-1}}\sum_{m=1}^{N_t}C_t(m)\zeta(2s-1,m/N_t),
\end{eqnarray}
where $\zeta(s,\alpha)$ denotes the Hurwitz zeta function. Now we truncate the contour $\text{Re}(s)=c$ at a height $R\in\mathbb{N}$. Next, we deform this truncated contour to a rectangle, $\mathcal{R}$ by enclosing the points $n=0, -1,\cdots, -R$ by choosing the side of the rectangle with $\text{Re}(s)=-R-1/2$. At this point, we just have to note that $\Gamma(s)$ has a simple pole at each $s=-n$ with residue $(-1)^n/n!$. Hence
\begin{eqnarray}\label{res}
\dfrac{1}{2i\pi}\int_{\mathcal{R}}(3\cdot 2^{t+2})^{s}\Gamma(s)L(2s-1,C_t)u^{-s}\;ds&=&\sum_{n=0}^R\displaystyle\mbox{Res}_{s=-n}\left[(3\cdot 2^{t+2})^{s}\Gamma(s)L(2s-1,C_t)u^{-s}\right]\notag\\&=&\sum_{n=0}^R\dfrac{(-1)^n}{n!(3\cdot 2^{t+2})^{n}}L(-2n-1,C_t)u^n.
\end{eqnarray}
To obtain the error caused by deformation and truncation of the contour, we use (\ref{Hu}) and the functional equation for the Hurwitz zeta function, and the asymptotics of the Gamma function to find that, for fixed $\text{Re}(s)$, $L(2s-1,C_t)$ grows at most like a polynomial in $|\text{Im}(s)|$ as $|\text{Im}(s)|\rightarrow \infty$ (see \cite[Theorem 12.24, pp. 272]{Apo}). 
Thus \eqref{res} and the estimates described above yield that the error is $O(u^{R+1/2})$. This error and equations (\ref{fu}) and (\ref{res}) imply as $R\rightarrow\infty$ that
\begin{eqnarray}\label{coeff2}
f(u)\sim -\dfrac{1}{2}\sum_{n=0}^\infty \dfrac{(-1)^n}{n!(3\cdot 2^{t+2})^{n}}L(-2n-1,C_t)u^n,\hspace{1cm} u\rightarrow 0^+.
\end{eqnarray}
On the other hand, we know from \cite[Prop. 2.4, pp. 7]{Be} that 
\begin{eqnarray}\label{coeff3}
f(u)\sim e^{\frac{-u(2^{t+1}-3)^2}{3\cdot 2^{t+2}}}\mathscr{F}_t(\zeta_p e^{-u}),\hspace{1cm}u\rightarrow 0^+,
\end{eqnarray}
whence (\ref{coeff1}) and (\ref{coeff2}) yield by comparing coefficients that
\begin{eqnarray}\label{coeff4}
c_{n,t}(\zeta_p)=-(-1)^{n}L(-2n-1,C_t)/2.
\end{eqnarray}
Using (\ref{LBer}) in (\ref{coeff4}), together with the fact that the function $\psi_t(m):=\zeta_p^{\frac{m^2-(2^{t+1}-3)^2}{3\cdot 2^{t+2}}}\chi_t(m)\;B_{2n+2}(m/N_t)$ equals $\psi_t(N_t-m)$, we obtain the required expression for $c_{n,t}(\zeta_p)$.

Next, we note that
\begin{eqnarray}\label{coeff5}
\mathscr{F}_t(\zeta_pe^{-u})&=&\left(e^{\frac{-u(2^{t+1}-3)^2}{3\cdot 2^{t+2}}}\mathscr{F}_t(\zeta_p e^{-u})\right)e^{\frac{u(2^{t+1}-3)^2}{3\cdot
2^{t+2}}}\notag\\
&=& \left(\sum_{m=0}^\infty \dfrac{c_{m,t}(\zeta_p)}{m!}\left(\dfrac{u}{3\cdot 2^{t+2}}\right)^m\right)\left(\sum_{k=0}^\infty \left(\dfrac{(2^{t+1}-3)^{2}}{3\cdot 2^{t+2}}\right)^k\dfrac{u^k}{k!}\right).
\end{eqnarray}
Hence comparing coefficients of $u^n$ on both sides of (\ref{coeff5}) we find
\begin{eqnarray}\label{coeff6'}
\dfrac{b_{n,t}(\zeta_p)}{n!}=\dfrac{(2^{t+1}-3)^{2n}}{(3\cdot 2^{t+2})^n}\sum_{\ell=0}^n \dfrac{c_{\ell,t}(\zeta_p)}{(2^{t+1}-3)^{2\ell}\;\ell!\;(n-\ell)!},
\end{eqnarray}
which yields the required expression of $b_{n,t}(\zeta_p)$ by putting the expression of $c_{n,t}(\zeta)$ in (\ref{coeff6'}).
\end{proof}
\begin{lem}\label{ls}
Let $a, b\in\mathbb{N}$ with $(p,ab)=1$ and $p^2\equiv 1\;(\emph{mod}\;b)$. Then the least non-negative integer $x_0$ satisfying $bx\equiv -a\;(\emph{mod}\;p)$ is
\begin{eqnarray*}
x_0=C_{a,b,p}-p\left\lfloor\dfrac{C_{a,b,p}}{p}\right\rfloor
\end{eqnarray*}
where 
\begin{eqnarray*}
C_{a,b,p}:=\dfrac{a(p^2-1)}{b}-p\left\lfloor\dfrac{ap}{b}\right\rfloor.
\end{eqnarray*}
In particular, when $a<b$ we have
\begin{eqnarray*}
x_0=C_{a,b,p}.
\end{eqnarray*}
\end{lem}
\begin{proof}
Since $(p,ab)=1$, we easily see that that the congruence $bx\equiv -a\;(\mbox{mod}\;p)$ has a non-zero unique solution mod $p$. Also the condition $p^2\equiv 1\;(\mbox{mod}\;b)$ implies that $C_{a,b,p}$ is an integer. Thus $x_0$ as defined in the lemma is an integer and is clearly a solution of the congruence. To show that $0\leq x_0\leq p-1$, we just have to note that $x_0$ is the remainder when dividing $C_{a,b,p}$ by $p$.

If $a<b$, we have
\begin{eqnarray}\label{ineq1}
C_{a,b,p}\geq \dfrac{a(p^2-1)}{b}-\dfrac{ap^2}{b}=-a/b\geq 0,
\end{eqnarray}
where the last inequality in (\ref{ineq1}) is due to the fact that $a<b$ and $C_{a,b,p}\in\mathbb{Z}$. On the other hand, we have
\begin{eqnarray}\label{ineq2}
C_{a,b,p}\leq \dfrac{a(p^2-1)}{b}-p\left(\dfrac{ap}{b}-1\right)=p-a/b\leq p-1,
\end{eqnarray}
where, once again, the last inequality in (\ref{ineq2}) is due to the fact that $a<b$ and $C_{a,b,p}\in\mathbb{Z}$. In view of (\ref{ineq1}) and (\ref{ineq2}) we see that $\left\lfloor\dfrac{C_{a,b,p}}{p}\right\rfloor=0$, and we are done.
\end{proof}
Next, let $C(n,i,j,p)$, $n\geq 0, \;0\leq i\leq p-1,\;0\leq j\leq n$ denote a Stirling like array of numbers (defined by Andrews-Sellers \cite{AS}) satisfying the following recurrence:
\begin{eqnarray*}
C(n+1,i,j,p)=(i+jp)C(n,i,j,p)+pC(n,i,j-1,p),
\end{eqnarray*}
and the initial value
\begin{eqnarray*}
C(0,i,0,p)=1.
\end{eqnarray*}
By defining a generalization of the signless Stirling numbers of the first kind, $s_1(n,j,m)$ for $0\leq j\leq n$ as follows:
\begin{eqnarray*}
\sum_{j=0}^ns_1(n,j,m)x^j=(x-m)(x-m+1)\cdots (x-m+n-1),
\end{eqnarray*}
Garvan \cite{G} proved the following result which we will require later:
\begin{lem}\label{Gar}
Let $i_0=(p^2-1)z-mp$. Suppose that
\begin{eqnarray*}
\sum_{\ell=0}^nC(n,i_0,\ell,p)A_1(\ell,m)=(-1)^nz^n\sum_{k=0}^n\binom{n}{k} X(k)
\end{eqnarray*}
for $n\geq 0$. Then for $n\geq 0$ we have 
\begin{eqnarray*}
A_1(n,m)=(-1)^n\sum_{k=0}^n\sum_{j=k}^n\binom{j}{k}s_1(n,j,m)p^{j-2k}X(k)z^j.
\end{eqnarray*}
\end{lem}
\begin{lem}\label{eqcoeff}
Let $p\geq 5$ be a prime and assume $0\leq j\leq p-1$ and $0\leq k\leq M-1\leq N-1$. Then
\begin{eqnarray*}
\alpha_t(p,N,j,k)=\alpha_t(p,M,j,k),
\end{eqnarray*}
where $\alpha_t(p,n,i,k)$ is defined in \emph{\eqref{sum1}}.
\end{lem}
\begin{proof}
Using (\ref{p-dis}) and orthogonality of roots of unity, it follows that
\begin{eqnarray}\label{eq4}
\mathcal{A}_{p,t}(N,j,q)=\dfrac{1}{p}\sum_{k=0}^{p-1}\zeta_p^{-jk}q^{-j/p}\mathscr{F}_t(\zeta_p^{k}q^{1/p};N).
\end{eqnarray}
Next, let us take $n\geq pM$ so that
\begin{eqnarray}\label{eq5}
(\zeta_p^k(1-q)^{1/p},\zeta_p^k(1-q)^{1/p})_n&=&\prod_{j=1}^n(1-\zeta_p^{kj}(1-q)^{j/p})\notag\\
&=&\prod_{\substack{1\leq j\leq n\\p\mid j}}(1-\zeta_p^{kj}(1-q)^{j/p})\prod_{\substack{1\leq j\leq n\\p\nmid j}}(1-\zeta_p^{kj}(1-q)^{j/p})\notag\\
&=&\prod_{1\leq j\leq n/p}(1-(1-q)^{j})\prod_{\substack{1\leq j\leq n\\p\nmid j}}(1-\zeta_p^{kj}(1-q)^{j/p})\notag\\
&=&\prod_{1\leq j\leq n/p}(jq+O(q^2))\prod_{\substack{1\leq j\leq n\\p\nmid j}}(1-\zeta_p^{kj}(1-q)^{j/p})=O(q^M),
\end{eqnarray}
since $n/p\geq M$. Thus from (\ref{trunc}), (\ref{eq4}) and \eqref{eq5} we have
\begin{eqnarray}
\mathcal{A}_{p,t}(pN-1,j,1-q)&=&\dfrac{1}{p}\sum_{k=0}^{p-1}\zeta_p^{-jk}(1-q)^{-j/p}\mathscr{F}_t(\zeta_p^{k}(1-q)^{1/p};pM-1)+O(q^M)\notag\\
&=&\mathcal{A}_{p,t}(pM-1,j,1-q)+O(q^M),
\end{eqnarray}
from which the lemma follows.
\end{proof}
\begin{lem}\label{setp}
The set of primes $p$ satisfying the quadratic congruence $p^2\equiv 1\;(\emph{mod}\;3\cdot 2^{t+2})$ are precisely the sets described in \emph{\textbf{(I)} - \textbf{(IV)}}.
\end{lem}
\begin{proof}
Let $p$ be a prime satisfying
\begin{eqnarray}\label{qcong}
p^2\equiv 1\;(\mbox{mod}\;3\cdot 2^{t+2}).
\end{eqnarray}
Using the division algorithm, we have $p=3\cdot 2^{t+1}k+j$ for some $k\in\mathbb{Z}$ and $0<j<3\cdot 2^{t+1}$ with $(j,6)=1$. It is clear that $p$ satisfies the congruence (\ref{qcong}) if and only if
\begin{eqnarray}\label{qcong1}
j^2\equiv 1\;(\mbox{mod}\;3\cdot 2^{t+2}).
\end{eqnarray}
Congruence (\ref{qcong1}) has a solution if and only if the following system of congruences has solutions
\begin{eqnarray}\label{qcong2}
j^2\equiv 1\;(\mbox{mod}\;3),\hspace{0.2cm}j^2\equiv 1\;(\mbox{mod}\;2^{t+2}).
\end{eqnarray}
Clearly, the first congruence in (\ref{qcong2}) has two solutions, namely
\begin{eqnarray}\label{qcong3}
j\equiv 1,\;2\;(\mbox{mod}\;3).
\end{eqnarray}
To solve the second congruence in (\ref{qcong2}), we use Hensel's lemma to lift a solution modulo $2^{t+1}$ to a solution modulo $2^{t+2}$. To this end, let us put $f(j)=j^2-1$. Then the second congruence in (\ref{qcong2}) is equivalent to 
\begin{eqnarray}\label{qcong4}
f(j)\equiv 0\;(\mbox{mod}\;2^{t+2}).
\end{eqnarray}
Plainly, we see that $f(1)=0$ which implies that $j=1$ is a solution of (\ref{qcong4}). Next, we have $f'(1)=2\equiv 0\;(\mbox{mod}\;2)$. Hence, using Hensel's lemma (see, for example, \cite[Chapter 5, pp. 120]{Apo}) we find two distinct solutions of (\ref{qcong4}) , namely
\begin{eqnarray}\label{qcong5}
j\equiv 1,\;1+2^{t+1}\;(\mbox{mod}\;2^{t+2}).
\end{eqnarray}
Combining \eqref{qcong3} and \eqref{qcong5}, we see that there are four possible solutions of \eqref{qcong2}. Namely, we have
\begin{eqnarray}
&&j\equiv 1\;(\mbox{mod}\;3),\hspace{0.2cm}j\equiv 1\;(\mbox{mod}\;2^{t+2}),\label{set1}\\
&&j\equiv 1\;(\mbox{mod}\;3),\hspace{0.2cm}j\equiv 1+2^{t+1}\;(\mbox{mod}\;2^{t+2}),\label{set2}\\
&&j\equiv 2\;(\mbox{mod}\;3),\hspace{0.2cm}j\equiv 1\;(\mbox{mod}\;2^{t+2}),\label{set3}\\
&&j\equiv 2\;(\mbox{mod}\;3),\hspace{0.2cm}j\equiv 1+2^{t+1}\;(\mbox{mod}\;2^{t+2}).\label{set4}
\end{eqnarray}
Clearly, the solution of (\ref{set1}) is $j\equiv 1\;(\mbox{mod}\;3\cdot 2^{t+2})$. Combining this with $p\equiv j\;(\mbox{mod}\;3\cdot 2^{t+1})$ we have
\begin{eqnarray*}
p\equiv 1\;(\mbox{mod}\;3\cdot 2^{t+1}).
\end{eqnarray*}
To solve the system of congruences in \eqref{set2} - \eqref{set4}, we use the Chinese remainder theorem. For that, we need to solve the following congruences:
\begin{eqnarray}
&&2^{t+2}\cdot m_1\equiv 1\;(\mbox{mod}\;3),\label{inv1}\\
&&3\cdot m_2\equiv 1\;(\mbox{mod}\;2^{t+2})\label{inv2},
\end{eqnarray}
for some integers $m_1$ and $m_2$. In order to solve \eqref{inv1} and \eqref{inv2}, we need to consider the following two cases:\\
\textbf{Case A}: If $t$ is even. In this case, we replace $t$ by $2t$. Then the solutions of \eqref{inv1} and \eqref{inv2} are given by
\begin{eqnarray}
&&m_1\equiv 1\;(\mbox{mod}\;3),\label{sole1}\\
&&m_2\equiv \dfrac{1+2^{2t+3}}{3}\;(\mbox{mod}\;2^{2t+2}).\label{sole2}
\end{eqnarray}
It is not hard to see that $(1+2^{2t+3})/3$ is indeed an integer less than $2^{2t+2}$. \\
\textbf{Case B}: If $t$ is odd. In this case, we replace $t$ by $2t+1$. Then the solutions of \eqref{inv1} and \eqref{inv2} are given by
\begin{eqnarray}
&&m_1\equiv 2\;(\mbox{mod}\;3),\label{solo1}\\
&&m_2\equiv \dfrac{1+2^{2t+3}}{3}\;(\mbox{mod}\;2^{2t+3}).\label{solo2}
\end{eqnarray}
Using the Chinese remainder theorem and congruences \eqref{sole1} - \eqref{solo2}, the solution of the system \eqref{set2} is obtained by
\begin{eqnarray}
j=\left\{\begin{array}{cc}
     (1+2^{2t+1})\cdot 3\cdot\dfrac{1+2^{2t+3}}{3}+1\cdot 2^{2t+2}\cdot 1,& \mbox{for \textbf{Case A}}, \\
     (1+2^{2t+1})\cdot 3\cdot\dfrac{1+2^{2t+3}}{3}+1\cdot 2^{2t+3}\cdot 2,& \mbox{for \textbf{Case B}}, 
\end{array}\right.
\end{eqnarray}
which when simplified modulo $3\cdot 2^{t+2}$ reduces to 
\begin{eqnarray}
j\equiv 1+3\cdot 2^{t+1}\equiv 3\cdot 2^{t+1}-1\;(\mbox{mod}\;3\cdot 2^{t+2}).
\end{eqnarray}
Combining this with $p\equiv j\;(\mbox{mod}\;3\cdot 2^{t+1})$, we get
\begin{eqnarray}
p\equiv 3\cdot 2^{t+1}-1\;(\mbox{mod}\;3\cdot 2^{t+1}).
\end{eqnarray}
In a similar manner, using the Chinese remainder theorem, we easily obtain the solutions $j\;(\mbox{mod}\;3\cdot 2^{t+2})$ of the systems \eqref{set3} and \eqref{set4}. Combining these solutions with the fact that $p\equiv j\;(\mbox{mod}\;3\cdot 2^{t+1})$, we precisely obtain the set $\mathcal{P}_3$ (resp. $\mathcal{P}_4$) defined by the congruences in \textbf{(III)} (resp. \textbf{(IV)}). 
\end{proof}
\begin{lem}\label{signpre}
For the sets $\mathcal{P}_1, \mathcal{P}_2, \mathcal{P}_3$ and $\mathcal{P}_4$ of primes described in \emph{\textbf{(I)} - \textbf{(IV)}}, we have
\begin{eqnarray*}
\chi_t(mp)=\left\{\begin{array}{cc}
    \chi_t(m), & \emph{if}\;p\in\mathcal{P}_1\cup\mathcal{P}_2, \\
     -\chi_t(m),& \emph{if}\;p\in\mathcal{P}_3\cup\mathcal{P}_4
\end{array}\right.
\end{eqnarray*}
for all $m$ such that $\chi_t(mp)\neq 0$.
\end{lem}
\begin{proof}
First, let us denote by $C_1$ and $C_2$ the following residue classes of integers.
\begin{eqnarray*}
C_1&:=&\{n\in\mathbb{Z}: n\equiv \pm (2^{t+1}-3)\;(\mbox{mod}\;3\cdot 2^{t+1})\}\\
C_2&:=&\{n\in\mathbb{Z}: n\equiv \pm (2^{t+1}+3)\;(\mbox{mod}\;3\cdot 2^{t+1})\}.
\end{eqnarray*}
Then it is clear that
\begin{eqnarray}
\chi_t(n)=\left\{
\begin{array}{ccc}
    1, &  \mbox{if}\;n\in C_1,\\
    -1, &  \mbox{if}\;n\in C_2,\\
    0,&\mbox{otherwise}.
\end{array}\right.
\end{eqnarray}
Let $p\in\mathcal{P}_1\cup\mathcal{P}_2$. Then we have
\begin{eqnarray}\label{pcong1}
p\equiv \pm 1\;(\mbox{mod}\;3\cdot 2^{t+1}).
\end{eqnarray}
Next, let $m$ be such that $\chi_t(mp)\neq 0$. Then we must have $mp\in C_1\cup C_2$. Since $p$ satisfies \eqref{pcong1}, we immediately see that
\begin{eqnarray*}
1)\;\;mp\in C_1\;\mbox{if and only if}\;m\in C_1,\hspace{0.5cm} 2)\;\;mp\in C_2\;\mbox{if and only if}\;m\in C_2.
\end{eqnarray*}
This is equivalent to saying that
\begin{eqnarray}\label{spre}
\chi_t(mp)=\chi_t(m),\;\mbox{for}\;p\in\mathcal{P}_1\cup\mathcal{P}_2. 
\end{eqnarray}
Now, let $p\in\mathcal{P}_3$. Then we have
\begin{eqnarray}\label{mp3}
p\equiv r_1(t),\;\;(\mbox{mod}\;3\cdot 2^{t+1}),\;r_{1}(t)=\left\{\begin{array}{cc}2^{t+1}-1,& t\;\mbox{even},\\
2^{t+1}+1,& t\;\mbox{odd}.
\end{array}\right.
\end{eqnarray}
First assume that $t$ is even, so that we replace $t$ by $2t$. Thus, if $m$ be such that $\chi_{2t}(mp)\neq 0$, then we must have $mp\in C_1\cup C_2$. Let $mp\in C_1$, then we have
\begin{eqnarray}\label{tevp3}
mp\equiv \pm(2^{2t+1}-3)\;(\mbox{mod}\;3\cdot 2^{2t+1}).
\end{eqnarray}
Since $p\in\mathcal{P}_3$ and $t$ is even, \eqref{mp3} and \eqref{tevp3} imply that
\begin{eqnarray}\label{inv3}
m(2^{2t+1}-1)\equiv \pm(2^{2t+1}-3)\;(\mbox{mod}\;3\cdot 2^{2t+1}).
\end{eqnarray}
Note that 
\begin{eqnarray}\label{someeq1}
(2^{2t+1}-1)(2^{2t+1}+3)=2^{2t+1}(2^{2t+1}+1)+(2^{2t+1}-3)\equiv (2^{2t+1}-3)\;(\mbox{mod}\;3\cdot 2^{2t+1}).
\end{eqnarray}
Thus, multiplying both sides of the congruence in \eqref{inv3} by $(2^{2t+1}+3)$ and using \eqref{someeq1} we get
\begin{eqnarray*}
m(2^{2t+1}-3)\equiv \pm (2^{2t+1}-3)(2^{2t+1}+3)\;(\mbox{mod}\;3\cdot 2^{2t+1}), 
\end{eqnarray*}
and since $(2^{2t+1}-3, 3\cdot 2^{2t+1})=1$, this yields
\begin{eqnarray}
m\equiv \pm(2^{2t+1}+3)\;(\mbox{mod}\;3\cdot 2^{2t+1}),
\end{eqnarray}
which implies that $m\in C_2$. Next, let $mp\in C_2$, that is
\begin{eqnarray}\label{todp3}
mp\equiv \pm(2^{2t+1}+3)\;(\mbox{mod}\;3\cdot 2^{2t+1}).
\end{eqnarray}
Since $p\in\mathcal{P}_3$ and $t$ is even, \eqref{mp3} and \eqref{todp3} imply that
\begin{eqnarray}\label{inv4}
m(2^{2t+1}-1)\equiv \pm(2^{2t+1}+3)\;(\mbox{mod}\;3\cdot 2^{2t+1}).
\end{eqnarray}
Note that 
\begin{eqnarray}\label{someeq2}
(2^{2t+1}-1)(2^{2t+1}-3)=2^{2t+1}(2^{2t+1}-5)+(2^{2t+1}+3)\equiv (2^{2t+1}+3)\;(\mbox{mod}\;3\cdot 2^{2t+1}).
\end{eqnarray}
Thus, multiplying both sides of the congruence in \eqref{inv4} by $(2^{2t+1}-3)$ and using \eqref{someeq2} we get
\begin{eqnarray*}
m(2^{2t+1}+3)\equiv \pm (2^{2t+1}-3)(2^{2t+1}+3)\;(\mbox{mod}\;3\cdot 2^{2t+1}), 
\end{eqnarray*}
and since $(2^{2t+1}+3, 3\cdot 2^{2t+1})=1$, this yields
\begin{eqnarray}
m\equiv \pm(2^{2t+1}-3)\;(\mbox{mod}\;3\cdot 2^{2t+1}),
\end{eqnarray}
which implies that $m\in C_1$. Combining all of the above, we see that for $t$ even, $mp\in C_1$ if and only if $m\in C_2$ and $mp\in C_2$ if and only if $m\in C_1$. In a similar way, it follows that for $t$ odd, $mp\in C_1$ if and only if $m\in C_2$ and $mp\in C_2$ if and only if $m\in C_1$. These imply that for primes $p\in\mathcal{P}_3$
\begin{eqnarray}\label{srev1}
\chi_t(mp)=-\chi_t(m),
\end{eqnarray}
for all $m$ such that $\chi_t(mp)\neq 0$. Using the same ideas as above, one shows for primes $p\in\mathcal{P}_4$ and $m$ such that $\chi_t(mp)\neq 0$ that
\begin{eqnarray}\label{srev2}
\chi_t(mp)=-\chi_t(m).
\end{eqnarray}
The result now follows from \eqref{spre}, \eqref{srev1} and \eqref{srev2}.
\end{proof}
Next, we prove a crucial result for $\mathscr{F}_t(q)$, similar to Garvan's result \cite[Theorem 1.12, pp. 5]{G} for $F(q)$. 
\begin{lem}\label{divcoeff}
Assume $t\geq 2$ and set $a=(2^{t+1}-3)^2, b=3\cdot 2^{t+2}$. Let $p\in\mathcal{P}_1\cup\mathcal{P}_2\cup\mathcal{P}_3\cup\mathcal{P}_4$ where $\mathcal{P}_1, \mathcal{P}_2, \mathcal{P}_3$ and $\mathcal{P}_4$ are as in \emph{\textbf{(I)} - \textbf{(IV)}}. Let $1\leq i_0\leq p-1$ be such that $b\cdot i_0\equiv -a\;(\emph{mod}\;p)$. Then we have
\begin{eqnarray*}
\alpha_t(p,n,i_0,k)=\left\{
\begin{array}{cc}
    p\cdot \xi_{p,\lfloor ap/b\rfloor+\lfloor C_{a,b,p}/p\rfloor,1,t}(n), & \emph{if}\;p\in\mathcal{P}_1\cup\mathcal{P}_2, \\
     -p\cdot \xi_{p,\lfloor ap/b\rfloor+\lfloor C_{a,b,p}/p\rfloor,1,t}(n), & \emph{if}\;p\in\mathcal{P}_3\cup\mathcal{P}_4, 
\end{array}\right.
\end{eqnarray*}
for $0\leq k\leq n-1$ and where $\xi_{r,s,N,t}(n)$ are defined as in \emph{\eqref{gFT2}} and $C_{a,b,p}$ is as in Lemma \emph{\ref{ls}}. Equivalently, we have
\begin{eqnarray*}
\mathcal{A}_{p,t}(pn-1,i_0,q)=\left\{\begin{array}{cc}
     pq^{\lfloor ap/b\rfloor+\lfloor C_{a,b,p}/p\rfloor}\mathscr{F}_t(q^p,pn-1)+(1-q)^n\beta_{p,t}(n,i_0,q),&  \emph{if}\;p\in\mathcal{P}_1\cup\mathcal{P}_2,\\
     -pq^{\lfloor ap/b\rfloor+\lfloor C_{a,b,p}/p\rfloor}\mathscr{F}_t(q^p,pn-1)+(1-q)^n\gamma_{p,t}(n,i_0,q),& \emph{if}\;p\in\mathcal{P}_3\cup\mathcal{P}_4, 
\end{array}\right.
\end{eqnarray*}
where $\beta_{p,t}(n,i_0,q), \gamma_{p,t}(n,i_0,q)\in\mathbb{Z}[q]$.
\end{lem}
\begin{proof}
As in \cite{AKL,AS, G}, for any $p$th root of unity $\zeta_p$, we have
\begin{eqnarray}\label{coeff6}
\ \ \ \ \ \ \ \ b_{n,t}(\zeta_p)=\left.\left(\dfrac{d}{du}\right)^n\mathscr{F}_t(\zeta_pe^{-u})\right|_{u=0}=(-1)^n\left.\left(q\dfrac{d}{dq}\right)^n\mathscr{F}_t(q)\right|_{q=\zeta_p}=(-1)^n\left.\left(q\dfrac{d}{dq}\right)^n\mathscr{F}_t(q;N)\right|_{q=\zeta_p},
\end{eqnarray}
for $N\geq p(n+1)-1$. Then from (\ref{p-dis}), we find using \cite[Lemma 2.4, pp. 5]{AS} that
\begin{eqnarray}\label{coeff7}
\left(q\dfrac{d}{dq}\right)^n\mathscr{F}_t(q;p(n+1)-1)=\sum_{j=0}^n\sum_{i=0}^{p-1}C(n,i,j,p)q^{i+jp}\mathcal{A}_{p,t}^{(j)}(p(n+1)-1,i,q^p).
\end{eqnarray}
We now extract coefficients of exponents congruent to $i$\;$(\mbox{mod}\;p)$ of $q$ using orthogonality of roots of unity in (\ref{coeff7}) to find that
\begin{eqnarray}\label{coeff8}
\ \ \ \ \ \ \sum_{j=0}^nC(n,i,j,p)q^{i+jp}\mathcal{A}_{p,t}^{(j)}(p(n+1)-1,i,q^p)&=&\dfrac{1}{p}\sum_{j=0}^{p-1}\zeta_p^{-ji}\left.\left(\left(q\dfrac{d}{dq}\right)^n\mathscr{F}_t(q;p(n+1)-1)\right)\right|_{q=\zeta_p^{j}q}.
\end{eqnarray}
At this point we put $q=\zeta_p$ on both sides of (\ref{coeff8}) to find using (\ref{coeff6}) that
\begin{eqnarray}\label{coeff9}
\ \ \ \ \sum_{j=0}^nC(n,i,j,p)\mathcal{A}_{p,t}^{(j)}(p(n+1)-1,i,1)=\dfrac{(-1)^n}{p}\sum_{j=0}^{p-1}\zeta_p^{-(j+1)i}b_{n,t}(\zeta_p^{j+1})=\dfrac{(-1)^n}{p}\sum_{j=0}^{p-1}\zeta_p^{-ji}b_{n,t}(\zeta_p^{j}).
\end{eqnarray}
Using Lemma \ref{Coeffs} in (\ref{coeff9}) and by orthogonality we obtain
\begin{eqnarray}\label{lincom}
\sum_{j=0}^nC(n,i,j,p)\mathcal{A}_{p,t}^{(j)}(p(n+1)-1,i,1)&=&\dfrac{(2^{t+1}-3)^{2n}(-1)^n}{p(3\cdot 2^{t+2})^n}\sum_{j=0}^{p-1}\zeta_p^{-ji}\sum_{\ell=0}^n\binom{n}{\ell}\sum_{k=0}^{p-1}\gamma_t(\ell,k)\zeta_p^{jk}\notag\\
&=&\dfrac{(2^{t+1}-3)^{2n}(-1)^n}{p(3\cdot 2^{t+2})^n}\sum_{\ell=0}^n\binom{n}{\ell}\sum_{k=0}^{p-1}\gamma_t(\ell,k)\sum_{j=0}^{p-1}\zeta_p^{j(k-i)}\notag\\
&=&\dfrac{(2^{t+1}-3)^{2n}(-1)^n}{(3\cdot 2^{t+2})^n}\sum_{\ell=0}^n\binom{n}{\ell}\gamma_t(\ell,i).
\end{eqnarray}
Let $a=(2^{t+1}-3)^2,\;b=3\cdot 2^{t+2}$. Then we clearly see that $a$ and $b$ satisfy the conditions in Lemma \ref{ls}. Also, Lemma \ref{setp} implies that $p^2\equiv 1\;(\mbox{mod}\;3\cdot 2^{t+2})$ if and only if $p\in\mathcal{P}_1\cup\mathcal{P}_2\cup\mathcal{P}_3\cup\mathcal{P}_4$. Thus, if we denote by $i_0$ the least solution of the congruence $b\cdot i_0\equiv -a\;(\mbox{mod}\;p)$. Then we have 
\begin{eqnarray}
i_0=C_{a,b,p}-p\left\lfloor\dfrac{C_{a,b,p}}{p}\right\rfloor.
\end{eqnarray}
Next, we note that 
\begin{eqnarray}
\dfrac{m^2-a}{b}\equiv i_0\;(\mbox{mod}\;p)\;\mbox{if and only if}\;m\equiv 0\;(\mbox{mod}\;p).
\end{eqnarray}
Hence we have
\begin{eqnarray}\label{gamma1}
\gamma_t(\ell,i_0)&=&\dfrac{(-1)^\ell N_t^{2\ell+1}}{(2\ell+2)a^\ell}\sum_{0\leq mp\leq N_t/2}\chi_t(mp)B_{2\ell+2}(mp/N_t)\notag\\
&=&\dfrac{(-1)^\ell N_t^{2\ell+1}}{(2\ell+2)a^{\ell}}\sum_{0\leq mp\leq N_t/2}\chi_t(mp)B_{2\ell+2}\left(\dfrac{m}{3\cdot 2^{t+1}}\right).
\end{eqnarray}
Using Lemma \ref{signpre} in \eqref{gamma1} we obtain:
\begin{eqnarray}\label{gamma2}
\gamma_t(\ell,i_0)=\left\{\begin{array}{cc}
     \dfrac{(-1)^\ell N_t^{2\ell+1}}{(2\ell+2)a^{\ell}}\displaystyle\sum_{0\leq m\leq 3\cdot 2^t}\chi_t(m)B_{2\ell+2}\left(\dfrac{m}{3\cdot 2^{t+1}}\right),& \mbox{if}\;p\in\mathcal{P}_1\cup\mathcal{P}_2, \\
     -\dfrac{(-1)^\ell N_t^{2\ell+1}}{(2\ell+2)a^{\ell}}\displaystyle\sum_{0\leq m\leq 3\cdot 2^t}\chi_t(m)B_{2\ell+2}\left(\dfrac{m}{3\cdot 2^{t+1}}\right),&  \mbox{if}\;p\in\mathcal{P}_3\cup\mathcal{P}_4.
\end{array}\right.
\end{eqnarray}
We prove the result for primes $p\in\mathcal{P}_1\cup\mathcal{P}_2$ as the other case for primes $p\in\mathcal{P}_3\cup\mathcal{P}_4$ is similar. Using Lemma \ref{Coeffs} with $\zeta_p=1$, we immediately see from (\ref{gamma2}) that 
\begin{eqnarray}\label{eq2}
\gamma_t(\ell,i_0)=\dfrac{p^{2\ell+1}}{a^\ell}c_{\ell,t}(1).
\end{eqnarray}
Apply Lemma \ref{Gar} in (\ref{lincom}) with $z=a/b$ and $m=\lfloor ap/b\rfloor+\lfloor C_{a,b,p}/p\rfloor$ with $i=i_0$ and then use (\ref{eq2}) to get
\begin{eqnarray}\label{Apt}
\mathcal{A}_{p,t}^{(n)}(p(n+1)-1,i_0,1)&=&(-1)^n\sum_{k=0}^n\sum_{j=k}^n\binom{j}{k}s_1(n,j,\lfloor ap/b\rfloor+\lfloor C_{a,b,p}/p\rfloor)p^{j-2k}
\dfrac{p^{2k+1}}{a^k}c_{k,t}(1)
\left(\dfrac{a}{b}\right)^j\notag\\
&=&p(-1)^n\sum_{k=0}^n\sum_{j=k}^n\binom{j}{k}s_1(n,j,\lfloor ap/b\rfloor+\lfloor C_{a,b,p}/p\rfloor)p^{j}\dfrac{c_{k,t}(1)}{a^{k}}
\left(\dfrac{a}{b}\right)^j.
\end{eqnarray}
From (\ref{sum1}), it is clear that
\begin{eqnarray}\label{Aptrel}
\alpha_t(p,n+1,i_0,n)=\dfrac{(-1)^n}{n!}\mathcal{A}_{p,t}^{(n)}(p(n+1)-1,i_0,1).
\end{eqnarray}
Thus (\ref{Apt}) and (\ref{Aptrel}) yield
\begin{eqnarray}\label{gf1}
\sum_{n=0}^\infty\alpha_t(p,n+1,i_0,n)q^n&=& p\sum_{n=0}^\infty\dfrac{q^n}{n!}\sum_{k=0}^n\sum_{j=k}^n\binom{j}{k}s_1(n,j,\lfloor ap/b\rfloor+\lfloor C_{a,b,p}/p\rfloor)p^{j}\dfrac{c_{k,t}(1)}{a^{k}}\left(\dfrac{a}{b}\right)^j\notag\\
&=&p\sum_{j=0}^\infty\sum_{k=0}^j\left(\sum_{n=j}^\infty s_1(n,j,\lfloor ap/b\rfloor+\lfloor C_{a,b,p}/p\rfloor)\dfrac{q^n}{n!}\right)\left(\dfrac{ap}{b}\right)^j\binom{j}{k}\dfrac{c_{k,t}(1)}{a^k}.
\end{eqnarray}
Next, we note that
\begin{eqnarray}\label{gf2}
\sum_{n=j}^\infty s_1(n,j,k)\dfrac{q^n}{n!}=(1-q)^k\dfrac{(-\log(1-q))^j}{j!}.
\end{eqnarray}
Hence, using (\ref{gf2}) in (\ref{gf1}) we get
\begin{eqnarray}\label{gf3}
\sum_{n=0}^\infty\alpha_t(p,n+1,i_0,n)q^n&=&p(1-q)^{\lfloor ap/b\rfloor+\lfloor C_{a,b,p}/p\rfloor}\sum_{j=0}^\infty\sum_{k=0}^j\dfrac{(-ap\log(1-q))^j}{b^jj!}\binom{j}{k}\dfrac{c_{k,t}(1)}{a^k}\notag\\
&=&p(1-q)^{\lfloor ap/b\rfloor+\lfloor C_{a,b,p}/p\rfloor}\sum_{j=0}^\infty\left(\sum_{k=0}^j\dfrac{1}{a^k}\binom{j}{k}c_{k,t}(1)\right)\dfrac{(-ap\log(1-q))^j}{b^jj!}.
\end{eqnarray}
It is easy to observe via the generating functions in (\ref{coeff1}) that 
\begin{eqnarray}\label{eq3}
b_{j,t}(1)=\left(\dfrac{a}{b}\right)^j\sum_{k=0}^j\dfrac{1}{a^k}\binom{j}{k}c_{k,t}(1).
\end{eqnarray}
Thus using (\ref{eq3}) in (\ref{gf3}) we get
\begin{eqnarray}
\sum_{n=0}^\infty\alpha_t(p,n+1,i_0,n)q^n=p(1-q)^{\lfloor ap/b\rfloor+\lfloor C_{a,b,p}/p\rfloor}\sum_{j=0}^\infty \dfrac{b_{j,t}(1)}{j!}(-p\log(1-q))^j,
\end{eqnarray}
which on using Lemma \ref{Coeffs} yields
\begin{eqnarray}\label{gcoeff1}
\sum_{n=0}^\infty\alpha_t(p,n+1,i_0,n)q^n&=&p(1-q)^{\lfloor ap/b\rfloor+\lfloor C_{a,b,p}/p\rfloor}\mathscr{F}_t(\exp(p\log(1-q)))\notag\\&=&p(1-q)^{\lfloor ap/b\rfloor+\lfloor C_{a,b,p}/p\rfloor}\mathscr{F}_t((1-q)^p).
\end{eqnarray}
The theorem now follows from Lemma \ref{eqcoeff}, and equations (\ref{gcoeff1}) and (\ref{gFT2}) with $(r,s,N)=(p,\lfloor ap/b\rfloor+\lfloor C_{a,b,p}/p\rfloor,1)$.
\end{proof}
We next record a crucial divisibility property of the polynomials $\mathcal{A}_{p,t}(N,i,q)$ which follows in a straightforward manner using \cite[Theorem 1.2, pp. 3]{AKL} and \cite[Prop. 2.4, pp. 7]{Be}.
\begin{lem}\label{Scott}
Let $u$ and $N$ be positive integers and $i\not\in S_{t,\chi_t}(u,1,0)$. Then
\begin{eqnarray*}
(q;q)_{\lambda(N,u)}\mid \mathcal{A}_{p,t}(N,i,q),
\end{eqnarray*}
where $\lambda(N,u)=\left\lfloor\frac{N+1}{u}\right\rfloor$. In particular, for any $1\leq k\leq \lambda(N,u)$ we have
\begin{eqnarray*}
(1-q^k)^{\lfloor\lambda(N,u)/k\rfloor}\mid \mathcal{A}_{p,t}(N,i,q).
\end{eqnarray*}
\end{lem}
Now, we require a result analogous to \cite[Lemma 3.3, pp. 8]{Straub}. 
\begin{lem}\label{higherdiv}
Let $\lambda, k\in\mathbb{N}$. Then for all integers $r$ and $m$ such that $m\geq \lambda$ we have
\begin{eqnarray*}
(1-(\zeta_k-q)^{krp})^m\equiv O(q^{\lambda-1+p(m-\lambda-1)})\;\;\emph{(mod\;$p^\lambda$)}.
\end{eqnarray*}
\end{lem}
\begin{proof}
Using the binomial expansion we have
\begin{eqnarray}
(\zeta_k-q)^{p}=\zeta_k^p+pqh_{1,\zeta_k}(q)-q^p,
\end{eqnarray}
where $h_{1,\zeta_k}(q)\in\mathbb{Z}[\zeta_k][q]$. Thus for any $r\in\mathbb{Z}$ 
\begin{eqnarray}
1-(\zeta_k-q)^{kpr}&=&1-(\zeta_k^p+pq\cdot h_{1,\zeta_k}(q)-q^p)^{kr}=pqh_{2,\zeta_k}(q)+q^ph_{3,\zeta_k}(q),
\end{eqnarray}
where $h_{2,\zeta_k}(q), h_{3,\zeta_k}(q)\in\mathbb{Z}[\zeta_k][[q]]$. Finally, it follows (as in \cite{Straub}) modulo $p^\lambda$ that 
\begin{eqnarray}
(1-(\zeta_k-q)^{kpr})^m\equiv O(q^{\lambda-1+p(m-\lambda-1)}).
\end{eqnarray}
\end{proof}
We end this section with Kummer's theorem \cite{Kum} and a couple of important applications that we will require in the proofs of the main results. 
\begin{theorem}[Kummer]\label{Kum}
Let $p$ be a prime, and $n, k$ integers such that $k\geq 0$. Then the $p$-adic valuation of $\binom{n}{k}$ is equal to the number of carries when adding $k$ and $n-k$.  
\end{theorem}
\begin{lem}\label{xicon1}
Let $p$ be a prime, $r$ an integer, $j\in\{1,2,\cdots, p-1\}$ and $t, N\in\mathbb{N}$. If $s$ is an integer such that
\begin{eqnarray}\label{c2}
\binom{s}{p^2-j}\equiv 0\;(\emph{mod}\;p),
\end{eqnarray}
then, for any $m, \lambda\in\mathbb{N}$ we have
\begin{eqnarray*}
\xi_{p^2r,s,N,t}(p^\lambda m-j)\equiv 0\;(\emph{mod}\;p^{\lambda-1}).
\end{eqnarray*}
\end{lem}
\begin{proof}
We note that
\begin{eqnarray*}
\mathscr{F}_t(q)=\sum_{n\geq 0}\left(\prod_{j=1}^n(1-q^j)\right)f_{n,t}(q),
\end{eqnarray*}
where $f_{n,t}(q)\in\mathbb{Z}[q]$. Thus we have
\begin{eqnarray}
\sum_{n\geq 0}\xi_{p^2r,s,N,t}(n)q^n&=&(\zeta_N-q)^s\sum_{n\geq 0}\prod_{j=1}^n(1-(\zeta_N-q)^{p^2rj})f_{n,t}((\zeta_N-q)^{p^2r})\notag\\
&=&\sum_{n\geq 0}c_{n}(\zeta_N-q)^{p^2rn+s},
\end{eqnarray}
for some $c_n\in\mathbb{Z}$. It then suffices to show for any integer $a$ that the coefficient of $q^{p^\lambda m-j}$ in $(\zeta_N-q)^{p^2a+s}$ vanishes modulo $p^{\lambda-1}$. Equivalently, it is enough to show, for any integer $a$, that
\begin{eqnarray}
\binom{p^2a+s}{p^\lambda m-j}\equiv 0\;(\mbox{mod}\;p^{\lambda-1}).
\end{eqnarray}
Let $n=p^2a+s,\;k=p^\lambda m-j$ and $k'=n-k$. Also let $k_0, k_1, \cdots$ and $k_0', k_1',\cdots$ be the $p$-adic digits of $k$ and $k'$ respectively. By Theorem \ref{Kum}, we know that the $p$-adic valuation of $\binom{n}{k}$ equals the number of carries when adding $k$ and $n-k$ in base $p$. It is clear that
\begin{eqnarray}\label{kexp1}
k=(p-j)+(p^{\lambda}-p)+p^{\lambda}(m-1)=(p-j)+(p-1)(p+p^2+\cdots+p^{\lambda-1})+p^{\lambda}(m-1).
\end{eqnarray}
Next, let $m-1=m_0+m_1p+\cdots$ be the $p$-adic expansion of $m-1$. Then (\ref{kexp1}) implies that $k_0=p-j$, and $k_1=k_2=\cdots=k_{\lambda-1}=p-1$. Since the digits $k_1, k_2,\cdots, k_{\lambda-1}$ have maximal value, it follows that, if a carry occurs when adding $k_0+k_1p$ and $k_0'+k_1'p$, there will be at least $\lambda-1$ carries when adding $k$ and $k'$. Thus the lemma follows from Theorem \ref{Kum}.

It only remains to see if there is a carry when adding $k_0+k_1p$ and $k_0'+k_1'p$ which, again by Theorem \ref{Kum}, is equivalent to condition (\ref{c2}) in the lemma and we are done.  
\end{proof}
The next result is Lemma 3.4 in \cite{Straub} and is an application of Kummer's Theorem \ref{Kum}. 
\begin{lem}\label{St1}
Let $p$ be a prime, $a$ an integer and $i\in\{0,1,\cdots,p-1\}$. If $j$ is an integer such that $0<j<p-i$, then
\begin{eqnarray*}
\binom{pa+i}{p^\lambda m-j}\equiv 0\;(\emph{mod}\;p^\lambda).
\end{eqnarray*}
\end{lem}
\section{Proofs of the main results}\label{PM}
\subsection{Proof of Theorem \ref{main1}}
Let $\zeta_N$ be an $N$th root of unity and $n\geq N\lambda$. Consider the following:
\begin{eqnarray}\label{disag1}
&&F((\zeta_N-q)^r;pn-1)=\sum_{i=0}^{p-1}(\zeta_N-q)^{ri}A_{p}(pn-1,i,(\zeta_N-q)^{rp})\\
&=&\sum_{\substack{0\leq i\leq p-1\\i\in S(p,1,0)}}(\zeta_N-q)^{ri}A_{p}(pn-1,i,(\zeta_N-q)^{rp})+\sum_{\substack{0\leq i\leq p-1\\i\not\in S(p,1,0)}}(\zeta_N-q)^{ri}A_{p}(pn-1,i,(\zeta_N-q)^{rp}).\notag
\end{eqnarray}
Using Theorem 1 in \cite{AK}, it follows for $i\not\in S(p,1,0)$ that
\begin{eqnarray}\label{Scott2}
A_{p}(pn-1,i,q)=(1-q^N)^{\lfloor n/N\rfloor}\beta_p(n,i,q),
\end{eqnarray}
where $\beta_p(n,i,q)\in\mathbb{Z}[q]$. Thus \eqref{Scott2} yields
\begin{eqnarray}\label{Scott3}
A_{p}(pn-1,i,(\zeta_N-q)^{rp})=(1-(\zeta_N-q)^{Nrp})^{\lfloor n/N\rfloor}\beta_{p,\zeta_N}(n,i,q),
\end{eqnarray}
where $\beta_{p,\zeta_N}(n,i,q)\in\mathbb{Z}[\zeta_N][[q]]$. Using Lemma \ref{higherdiv} in \eqref{Scott3} we get
\begin{eqnarray}\label{Scott4}
A_p(pn-1,i,(\zeta_N-q)^{rp})\equiv O(q^{p\lfloor n/N\rfloor-(p-1)(\lambda-1)})\;(\mbox{mod}\;p^\lambda).
\end{eqnarray}
Thus \eqref{disag1} and \eqref{Scott4} imply that
\begin{eqnarray}\label{tail1}
F((\zeta_N-q)^r;pn-1)&\equiv &\sum_{\substack{0\leq i\leq p-1\\i\in S(p,1,0)}}(\zeta_N-q)^{ri}A_p(pn-1,i,(\zeta_N-q)^{rp})\notag\\&&+\;O(q^{p\lfloor n/N\rfloor-(p-1)(\lambda-1)})\;\;(\mbox{mod}\;p^\lambda). 
\end{eqnarray}
Choosing $n$ large enough, it now suffices to show, in view of \eqref{tail1} that the coefficient of $q^{p^\lambda m-j}$ in 
\begin{eqnarray}\label{aeq1}
(\zeta_N-q)^{ri+s}A_p(pn-1,i,(\zeta_N-q)^{rp})
\end{eqnarray}
vanishes modulo $p^\lambda$ for all $i\in S(p,1,0)$. First, let $j\in\left\{1,2,\cdots, p-1-\max{S(p,r,s)}\right\}$. Then since $A_p(pn-1,i,q)\in\mathbb{Z}[q]$, it suffices to show that the coefficient of $q^{p^\lambda m-j}$ in
\begin{eqnarray}
(\zeta_N-q)^{ap+ri+s}
\end{eqnarray}
vanishes modulo $p^\lambda$ for all $a\in\mathbb{Z}$. Since $p\nmid r$, we have $ri+s\equiv i'\;(\mbox{mod}\;p)$ for some $i'\in S(p,r,s)$. Thus, $ap+ri+s=a'p+i'$ where $a'\in\mathbb{Z}$ and $j<p-i'$. Thus it suffices to show that 
\begin{eqnarray}\label{bin1}
\binom{a'p+i'}{p^\lambda m-j}\equiv 0\;(\mbox{mod}\;p^\lambda).
\end{eqnarray}
Using Lemma \ref{St1}, we immediately see that (\ref{bin1}) is satisfied and we are done.

Next, suppose $(p,r,s)$ satisfies condition \ref{c1} and that $j\in\left\{1,2,\cdots, p-1-\max{S^*(p,r,s)}\right\}$. Let $i_0\in S(p,1,0)$ be such that $24i_0\equiv -1\;(\mbox{mod}\;p)$. By Lemma 2.2 in \cite{Straub} we have
\begin{eqnarray}
\mbox{}\\ A_p(pn-1,i_0,(\zeta_N-q)^{rp})=\left(\dfrac{12}{p}\right)p(\zeta_N-q)^{rp\lfloor p/24\rfloor}F((\zeta_N-q)^{rp^2},pn-1)+O((1-(\zeta_N-q)^{rp})^n).\notag
\end{eqnarray}
If $N\mid rp$, we easily see that $O((1-(\zeta_N-q)^{rp})^n)=O(q^n)$ and thus for large $n$ we can ignore this term. Thus it suffices to show that the coefficient of $q^{p^\lambda m-j}$ in 
\begin{eqnarray}
\left(\dfrac{12}{p}\right)p(\zeta_N-q)^{ri_0+s+rp\lfloor p/24\rfloor}F((\zeta_N-q)^{rp^2})
\end{eqnarray}
vanishes modulo $p^\lambda$. This is equivalent to the fact that 
\begin{eqnarray}\label{cong1}
\xi_{p^2r,ri_0+s+rp\lfloor p/24\rfloor,N}\equiv 0\;(\mbox{mod}\;p^{\lambda-1}). 
\end{eqnarray}
Congruence \ref{cong1} is essentially a consequence of Lemma 3.2 in \cite{Straub}, Kummer's Theorem \ref{Kum} and the condition that $\mbox{digit}_1(s-r/24;p)\neq p-1$. Thus, it only remains to show the coefficient of $q^{p^\lambda m-j}$ in (\ref{aeq1}) vanishes modulo $p^\lambda$ for $i\in S(p,1,0)\setminus\{i_0\}$, which follows in a similar manner. This completes the proof of the theorem.
\subsection{Proof of Theorem \ref{main2}}
Let $\zeta_N$ be an $N$th root of unity and $n\geq N\lambda$. Consider the following:
\begin{eqnarray}\label{disag3}
&&\mathscr{F}_t((\zeta_N-q)^r;pn-1)=\sum_{i=0}^{p-1}(\zeta_N-q)^{ri}\mathcal{A}_{p,t}(pn-1,i,(\zeta_N-q)^{rp})\\
&=&\sum_{\substack{0\leq i\leq p-1\\i\in S_{t,\chi_t}(p,1,0)}}(\zeta_N-q)^{ri}\mathcal{A}_{p,t}(pn-1,i,(\zeta_N-q)^{rp})+\sum_{\substack{0\leq i\leq p-1\\i\not\in S_{t,\chi_t}(p,1,0)}}(\zeta_N-q)^{ri}\mathcal{A}_{p,t}(pn-1,i,(\zeta_N-q)^{rp}).\notag
\end{eqnarray}
Using Lemma \ref{Scott}, it follows for $i\not\in S_{t,\chi_t}(p,1,0)$ that
\begin{eqnarray}\label{Scott22}
\mathcal{A}_{p,t}(pn-1,i,q)=(1-q^N)^{\lfloor n/N\rfloor}\beta_{p,t}(n,i,q),
\end{eqnarray}
where $\beta_{p,t}(n,i,q)\in\mathbb{Z}[q]$. Thus \eqref{Scott22} yields
\begin{eqnarray}\label{Scott33}
\mathcal{A}_{p,t}(pn-1,i,(\zeta_N-q)^{rp})=(1-(\zeta_N-q)^{Nrp})^{\lfloor n/N\rfloor}\beta_{p,\zeta_N,t}(n,i,q),
\end{eqnarray}
where $\beta_{p,\zeta_N,t}(n,i,q)\in\mathbb{Z}[\zeta_N][[q]]$. Using Lemma \ref{higherdiv} in \eqref{Scott33} we get
\begin{eqnarray}\label{Scott44}
\mathcal{A}_{p,t}(pn-1,i,(\zeta_N-q)^{rp})\equiv O(q^{p\lfloor n/N\rfloor-(p-1)(\lambda-1)})\;(\mbox{mod}\;p^\lambda).
\end{eqnarray}
Thus \eqref{disag3} and \eqref{Scott44} imply that
\begin{eqnarray}\label{tail3}
\mathscr{F}_t((\zeta_N-q)^r;pn-1)&\equiv &\sum_{\substack{0\leq i\leq p-1\\i\in S_{t,\chi_t}(p,1,0)}}(\zeta_N-q)^{ri}\mathcal{A}_{p,t}(pn-1,i,(\zeta_N-q)^{rp})\notag\\&&+\;O(q^{p\lfloor n/N\rfloor-(p-1)(\lambda-1)})\;\;(\mbox{mod}\;p^\lambda). 
\end{eqnarray}
Choosing $n$ large enough, it now suffices to show, in view of \eqref{tail3} that the coefficient of $q^{p^\lambda m-j}$ in 
\begin{eqnarray}\label{aeq3}
(\zeta_N-q)^{ri+s}\mathcal{A}_{p,t}(pn-1,i,(\zeta_N-q)^{rp})
\end{eqnarray}
vanishes modulo $p^\lambda$ for all $i\in S_{t,\chi_t}(p,1,0)$. First, let $j\in\left\{1,2,\cdots, p-1-\max{S_{t,\chi_t}(p,r,s)}\right\}$. Then since $\mathcal{A}_{p,t}(pn-1,i,q)\in\mathbb{Z}[q]$, it suffices to show that the coefficient of $q^{p^\lambda m-j}$ in
\begin{eqnarray}
(\zeta_N-q)^{ap+ri+s}
\end{eqnarray}
vanishes modulo $p^\lambda$ for all $a\in\mathbb{Z}$. Since $p\nmid r$, we have $ri+s\equiv i'\;(\mbox{mod}\;p)$ for some $i'\in S_{t,\chi_t}(p,r,s)$. Thus, $ap+ri+s=a'p+i'$ where $a'\in\mathbb{Z}$ and $j<p-i'$. Thus it suffices to show that 
\begin{eqnarray}\label{bin3}
\binom{a'p+i'}{p^\lambda m-j}\equiv 0\;(\mbox{mod}\;p^\lambda).
\end{eqnarray}
Using Lemma \ref{St1}, we immediately see that (\ref{bin3}) is satisfied and we are done.

Next, suppose $(p,r,s)$ satisfies condition \ref{c4} and that $j\in\left\{1,2,\cdots, p-1-\max{S^*_{t,\chi_t}(p,r,s)}\right\}$. Let $p\in\mathcal{P}_1\cup\mathcal{P}_2\cup\mathcal{P}_3\cup\mathcal{P}_4$. Let $i_0\in S_{t,\chi_t}(p,1,0)$ be such that $b\cdot i_0\equiv -a\;(\mbox{mod}\;p)$. By Lemma \ref{divcoeff} and (\ref{sum1}) we have
\begin{eqnarray}\label{divcoeff1}
\ \ \ \ \ \ \ \ \ \mathcal{A}_{p,t}(pn-1,i_0,q)=\left\{\begin{array}{cc}
     pq^{\lfloor ap/b\rfloor+\lfloor C_{a,b,p}/p\rfloor}\mathscr{F}_t(q^p,pn-1)+(1-q)^n\beta_{p,t}(n,i_0,q),&  \mbox{if}\;p\in\mathcal{P}_1\cup\mathcal{P}_2,\\
     -pq^{\lfloor ap/b\rfloor+\lfloor C_{a,b,p}/p\rfloor}\mathscr{F}_t(q^p,pn-1)+(1-q)^n\gamma_{p,t}(n,i_0,q),& \mbox{if}\;p\in\mathcal{P}_3\cup\mathcal{P}_4, 
\end{array}\right.
\end{eqnarray}
where $\beta_{p,t}(n,i_0,q), \gamma_{p,t}(n,i_0,q)\in\mathbb{Z}[q]$. We will prove the result for $p\in\mathcal{P}_1\cup\mathcal{P}_2$ since the other case for primes $p\in\mathcal{P}_3\cup\mathcal{P}_4$ is exactly similar. Thus, equation \eqref{divcoeff1} yields
\begin{eqnarray*}
\mathcal{A}_{p,t}(pn-1,i_0,(\zeta_N-q)^{rp})=p(\zeta_N-q)^{rp(\lfloor ap/b\rfloor+\lfloor C_{a,b,p}/p\rfloor)}\mathscr{F}_t((\zeta_N-q)^{rp^2},pn-1)+O((1-(\zeta_N-q)^{rp})^n).
\end{eqnarray*}
If $N\mid rp$, we easily see that $O((1-(\zeta_N-q)^{rp})^n)=O(q^n)$ and thus for large $n$ we can ignore this term. Thus it suffices to show that the coefficient of $q^{p^\lambda m-j}$ in 
\begin{eqnarray}
p(\zeta_N-q)^{ri_0+s+rp(\lfloor ap/b\rfloor+\lfloor C_{a,b,p}/p\rfloor)}\mathscr{F}_t((\zeta_N-q)^{rp^2})
\end{eqnarray}
vanishes modulo $p^\lambda$. This is equivalent to the fact that 
\begin{eqnarray}\label{cong3}
\xi_{p^2r,ri_0+s+rp(\lfloor ap/b\rfloor+\lfloor C_{a,b,p}/p\rfloor),N,t}\equiv 0\;(\mbox{mod}\;p^{\lambda-1}).
\end{eqnarray}
Using Lemma \ref{xicon1} we immediately see that \eqref{cong3} is satisfied, if we can show that
\begin{eqnarray}\label{equiv1}
\binom{ri_0+s+rp(\lfloor ap/b\rfloor+\lfloor C_{a,b,p}/p\rfloor)}{p^2-j}\equiv 0\;(\mbox{mod}\;p).
\end{eqnarray}
It is clear that the $p$-adic expansion of $-a/b$ has the form $i_0+i_1p+i_2p^2+\cdots$ where $i_1=\lfloor ap/b\rfloor+\lfloor C_{a,b,p}/b\rfloor\;(\mbox{mod}\;p)$. Indeed,
\begin{eqnarray}\label{firdig}
\mbox{digit}_1\left(-\dfrac{a}{b};p\right)=\mbox{digit}_1\left(\dfrac{a(p^2-1)}{b};p\right)=\dfrac{1}{p}\left(\dfrac{a(p^2-1)}{b}-i_0\right)\;(\mbox{mod}\;p)=\lfloor ap/b\rfloor+\lfloor C_{a,b,p}/p\rfloor\;(\mbox{mod}\;p).\notag\\
\end{eqnarray}
Thus, since $p\nmid r$, writing $\lfloor ap/b\rfloor+\lfloor C_{a,b,p}/p\rfloor=p\cdot k+i_1$ for some $k\in\mathbb{Z}$, it follows from \eqref{firdig} that
\begin{eqnarray}\label{equiv2}
ri_0+s+rp(\lfloor ap/b\rfloor+\lfloor C_{a,b,p}/b\rfloor)=s+r(i_0+p\cdot i_1)+krp^2\equiv s-ar/b\;(\mbox{mod}\;p^2).
\end{eqnarray}
Combining \eqref{equiv1} and \eqref{equiv2}, we only need to prove that
\begin{eqnarray}
\binom{s-ar/b}{p^2-j}\equiv 0\;(\mbox{mod}\;p),
\end{eqnarray}
which follows from Theorem \ref{Kum}, since $\mbox{digit}_1(p^2-j;p)=p-1$, but $\mbox{digit}_1(s-ar/b;p)<p-1$, by assumption. This completes the proof.
\subsection{Proof of Corollary \ref{main3}}
This result corresponds to the case $t=2$ in Theorem \ref{main2}. However, in this case we note that the primes $p\in\mathcal{P}_1\cup\mathcal{P}_2\cup\mathcal{P}_3\cup\mathcal{P}_4$ are precisely given by the congruences $p\equiv \pm 1, \pm 7\;(\mbox{mod}\;24)$, and we are done.
\section{Comments and Conclusion}
In an upcoming work with Robert Osburn, we show that for all $t\geq 2$, the
function $\mathscr{F}_t(q)$ (multiplied with a suitable power of $q$) is a quantum modular
form of weight $3/2$ in a certain subgroup of $SL_2(\mathbb{Z})$.

At this point, several open questions could be raised. First, it looks very likely that the methods used in this work could be used to prove congruences of coefficients at roots of unity related to Kontsevich-Zagier series associated to the family of torus knots $T(2, 2t+1)$, $t\geq 1$ (see \cite[Sec. 4]{Be}). Secondly, one could also investigate the asymptotic properties for the numbers $\xi_{r,s,N,t}(n)$. Finally, for further open questions, one should consult \cite{AS, Be,G, Straub}. 

\section*{Acknowledgement}
My research was supported by grant SFB F50-06 of the Austrian Science Fund (FWF). 

The author would like to thank George Andrews, Robert Osburn, Peter Paule and Armin Straub for their support and encouragement in carrying out this work. In particular, the author thanks Robert Osburn for describing how the Fishburn and generalized Fishburn numbers relate to torus knots, suggesting this more general study and sending the preprint \cite{Be}. The author also thanks him and Peter Paule for all the comments in an earlier draft of this work which improved exposition.

\end{document}